\documentclass[10pt,journal,compsoc,twocolumn]{IEEEtran}
\usepackage{url}
\usepackage{enumitem} 
\usepackage{amsfonts}
\usepackage{amsmath}
\usepackage{amssymb}
\usepackage{color}
\usepackage{wrapfig}
\usepackage{graphicx} 
\usepackage{epstopdf}
\usepackage{bm}
\usepackage{xcolor}
  \usepackage{amsthm}
\usepackage{dcolumn}
\usepackage{bm}
\usepackage{hyperref}
\usepackage{dsfont}
\usepackage[caption=false]{subfig}
\usepackage{float}
\usepackage{microtype}
\usepackage{lipsum}
\usepackage{cuted}

\newtheorem{theorem}{Theorem}

\usepackage{dblfloatfix}
\begin{document}
\title{Optimal Scheduling of Dengue Vector Control}

\author{ Aram Vajdi$\ ^{a}$, Lee W. Cohnstaedt$\ ^{b,*}$, Caterina M. Scoglio$\ ^{c}$, Heman Shakeri$\ ^{a}$

\IEEEcompsocitemizethanks{\IEEEcompsocthanksitem $^{*}$ Corresponding author. E-mail address: lee.cohnstaedt@usda.gov (L.W. Cohnstaedt). 
\protect\\$^{a}$ School of Data Science, University of Virginia, Charlottesville, Virginia, USA
\protect\\ $^{b}$ United States Department of Agriculture, Agricultural Research Service, Foreign Arthropod-Borne Animal Diseases Research Unit, Manhattan, Kansas, USA
\protect\\$^{c}$ Department of Electrical and Computer Engineering, Kansas State University,
Manhattan, KS, USA 

}     
}

\IEEEtitleabstractindextext{%
\begin{abstract}
Dengue transmission is shaped by the population dynamics of the \textit{Aedes aegypti} mosquito, making vector control a central strategy for disease mitigation. The impact of interventions such as larvicide, adulticide, and breeding-site reduction depends critically on their timing under fluctuating environmental conditions. We build on a high-fidelity, non-Markovian mechanistic model of the \textit{Aedes} life cycle that captures stage-structured, temperature-dependent developmental delays, and mortality, and extend it to incorporate multiple vector control measures. Rather than using continuous abstract control amplitudes as in standard optimal control formulations, we introduce intervention-specific temporal profiles that better reflect operational practice. We then develop an adjoint-based gradient descent framework to compute the optimal timing of a sequence of interventions by minimizing the time-dependent dengue reproduction number, $R_{0}$. Numerical simulations based on seasonal temperature data from Miami, Florida, show that appropriately timed combinations of interventions can substantially suppress transmission risk, with outcomes strongly influenced by seasonal temperature variation and intervention duration. We further propose embedding the resulting optimization framework within a Model Predictive Control architecture, yielding a closed-loop approach for real-time, surveillance-driven vector management under environmental and operational uncertainty.
\end{abstract}

}

\maketitle

\IEEEdisplaynontitleabstractindextext

%
\IEEEpeerreviewmaketitle

\IEEEraisesectionheading{\section{Introduction}\label{sec:introduction}}
Dengue is a mosquito-borne viral disease that has become a growing public health concern in many tropical and subtropical regions of the world~\cite{bhatt2013global}. The virus is primarily transmitted by the \textit{Aedes aegypti} mosquito, which thrives in urban environments and breeds in bodies of standing water common in such settings. 
During the gonotrophic cycle, female \textit{Aedes aegypti} mosquitoes feed on blood sources, making them highly efficient vectors for dengue transmission in densely populated areas. 
In the absence of broadly effective vaccines or antiviral treatments, most public health efforts have focused on vector control strategies aimed at reducing mosquito populations or limiting human--mosquito contact, including larval habitat reduction and the application of larvicides and adulticides~\cite{bowman2016dengue, achee2015critical, kalimuddin2025dengue}. However, the effectiveness of these interventions varies widely depending on the type of control measure and how it is implemented~\cite{alvarado2017assessing, samuel2017community, wilke2021effectiveness}, and some findings indicate that effective dengue control requires an optimal combination of multiple approaches~\cite{horstick2018building, jaffal2023current}. 

In this work, we investigate the optimal scheduling of vector control strategies using a high-fidelity mechanistic model of the \textit{Aedes} mosquito population. The mosquito life cycle is strongly influenced by environmental factors such as temperature and available breeding sites, with temperature governing key developmental and survival rates. Therefore, accurately capturing the effects of temporal variations in these environmental drivers is essential for modeling population dynamics and determining optimal control strategies. A growing body of literature has developed mechanistic compartmental models to describe dengue transmission dynamics and the population dynamics of \textit{Aedes aegypti} by incorporating entomological and environmental parameters into dynamic mathematical frameworks~\cite{focks1993dynamic, yang2011follow, yang2009assessing, yang2009assessing2, mordecai2017detecting, pliego2017seasonality, kilicman2021development}. Other studies have extended these compartmental models to include stochasticity, spatial structure, and additional factors such as vector mobility and human behavior~\cite{demers2020managing, becerra2026stochastic}. In addition, statistical and machine learning approaches have incorporated climatic variables to model dengue risk and mosquito abundance, with spatiotemporal data shown to significantly improve predictive accuracy~\cite{aswi2019bayesian, dom2025fine, rocha2022machine, chen2023linking}. To compute the temporal mosquito population, we use a model we recently developed to capture the non-Markovian nature of the mosquito life cycle~\cite{vajdi2024assessing}. The model accounts for habitat temperature through temperature-dependent development rates obtained from experimental studies. In contrast to standard Markovian formulations, it allows transitions to depend on residence time within each life stage, providing a more faithful representation of developmental delay and its effect on mosquito population dynamics under changing environmental conditions.

Typically, dengue optimal-control formulations in the literature represent interventions as generic bounded control functions. In contrast, we model mosquito control through intervention-specific temporal profiles that more closely reflect how these measures are implemented in practice. In particular, larvicide and adulticide are introduced as additional time-dependent mortality terms acting on the larval and adult stages, respectively, each characterized by a prescribed efficacy and finite duration of action. Breeding-site elimination, on the other hand, is represented as a transient reduction in environmental carrying capacity, followed by gradual recovery as breeding sites reappear over time. Accordingly, rather than optimizing arbitrary control amplitudes, our objective is to determine the optimal timing of a sequence of operationally realistic intervention events by identifying the optimal placement of the leading edges of these control profiles. More importantly, our formulation is built on a high-fidelity, temperature-driven, non-Markovian model of the mosquito life cycle, in which development across the egg, larval, pupal, and gonotrophic stages is accurately represented through sequential auxiliary substates. To determine the optimal timing of a sequence of interventions, we employ an adjoint-based gradient-descent framework to minimize the cumulative time-dependent dengue reproduction number, which we use as a proxy for transmission risk.
 
Previous dengue-control studies have typically relied on Markovian host--vector formulations, in which the mosquito life cycle is often reduced to a few aggregated compartments and interventions enter the system as continuous time-dependent control variables. In~\cite{rodrigues2012modeling}, larvicide, adulticide, and mechanical control are introduced as optimal-control functions in a six-compartment dengue model, and the results show that adulticide has the strongest short-term effect on lowering the basic reproduction number, while larval control becomes more important over longer horizons. In~\cite{pliego2020control}, season-dependent control is studied using chemical interventions in the hot season and mechanical breeding-site reduction in the cool season, with quadratic cost functionals and successive optimal-control problems solved over alternating seasonal intervals. In contrast to our framework, \cite{de2023multiobjective} is based on a climate-dependent epidemiological transmission model that incorporates temperature, precipitation, and humidity into the disease dynamics and formulates the intervention problem as a multiobjective evolutionary optimization task. In that work, the optimization variables include control-related parameters such as the magnitudes and initiation times of larvicide- and adulticide-type interventions. By contrast, our work is built on a high-fidelity, temperature-driven, non-Markovian mosquito life-cycle model, and the optimization is carried out through an adjoint-based gradient framework to determine the optimal timing of intervention-specific control profiles.
In~\cite{abidemi2024optimal}, the authors developed an optimal control model for dengue transmission in a coupled human--mosquito compartmental system that includes asymptomatic, isolated, and vigilant human classes. The study considers three time-dependent interventions: personal protection against mosquito bites, treatment of infected individuals, and insecticide spraying to reduce the mosquito population. They use optimal control theory to identify efficient and cost-effective intervention combinations. However, the model does not incorporate temperature, seasonality, or other environmental forcing, focusing instead on intervention strategies within a standard epidemiological framework. Another study~\cite{arias2020biological} investigated the optimal control of \textit{Aedes aegypti} using a combination of chemical interventions and biological control within a predator--prey framework. The model includes mosquito aquatic and adult stages, as well as a predator population feeding on immature mosquitoes. However, it assumes constant entomological rates and does not account for temperature or seasonal variation. In~\cite{agusto2018optimal}, the authors consider a dengue transmission model that includes adult mosquitoes and investigate vaccination and insecticide application as control measures. Using optimal control theory, they minimize disease burden and intervention costs in a setting where the entomological parameters are assumed to be constant. In a similar spirit, Miah et al.~\cite{miah2025optimal} and Herdicho et al.~\cite{herdicho2025optimal} developed optimal-control models for dengue transmission that examine combinations of interventions to reduce disease burden and implementation costs. While Miah et al.\ considered a human--vector SEIR--SEI framework with four control strategies (public awareness, treatment, adulticide, and larval-source reduction), Herdicho et al.\ focused on a sex-structured dengue hemorrhagic fever model with fumigation and exposure-reduction measures against mosquito bites. In~\cite{wijaya2014optimal}, a model was developed based on an indoor--outdoor life-cycle structure, with separate indoor and outdoor egg and larval compartments and a shared adult mosquito compartment. The control measures target eggs and larvae in indoor breeding sites as well as adult mosquitoes, and optimal control theory was used to minimize mosquito abundance while also minimizing control cost.

Compared with previous studies, we shift the focus from optimizing continuous control amplitudes in standard compartmental dengue-transmission models to optimizing the timing and scheduling of a sequence of predetermined interventions (namely larvicide application, adulticide application, and habitat elimination) under changing environmental conditions in which temperature modulates the entomological rates governing the \textit{Aedes aegypti} life cycle and, consequently, mosquito population dynamics. To describe mosquito population dynamics under these control measures, we use a non-Markovian development model that more accurately captures the transition time between life stages while accounting for temperature variation in the habitat. We then develop a gradient-based algorithm, derived from the adjoint method, to determine the optimal scheduling of the control measures to minimize the risk of dengue transmission. We have also implemented the algorithm described in this paper, and the code is available upon request from the authors. The rest of the paper is organized as follows. In Section~2, we first discuss the non-Markovian model for mosquito population dynamics under a sequence of control measures and then derive the optimization framework based on the adjoint method. In Section~3, we demonstrate the application of the framework by presenting numerical results using temperature data for Miami, Florida, and determine the optimal scheduling of various combinations of control measures under different levels of effectiveness. Finally, in Section~4, we discuss the integration of our framework into a Model Predictive Control architecture for real-time, surveillance-driven vector management, and outline directions for future work.

\section{Mosquito Life-Cycle Model and Transmission Risk} 
Transmission of vector-borne diseases such as dengue depends fundamentally on the population of mosquitoes capable of carrying and transmitting the virus. Dengue is transmitted to humans primarily by the bite of an infected adult female \textit{Aedes} mosquito, which acquires the virus by feeding on an infected person. Once infected, the mosquito remains infectious for the rest of its life and can continue transmitting the virus during subsequent blood meals. Therefore, when factors such as human immunity, mobility patterns, and human--mosquito contact rates, which are highly location-dependent, are not considered, mosquito-related factors such as the adult female mosquito population, biting rate, and extrinsic incubation period become crucial in determining dengue transmission risk. In fact, as shown in~\cite{vajdi2024assessing}, for a homogeneously mixed susceptible population, the basic reproduction number for dengue can be expressed as
\begin{equation}\label{reproF}
R_0=\frac{A}{N_H}n_B^2\gamma_{a,e}^2\phi_{HV}\phi_{VH} \gamma_{a,d}^{-1} \eta_H^{-1}\left(1+\frac{\gamma_{a,d}}{\gamma_V}\right)^{-1},
\end{equation} 
where $A$ is the adult female mosquito population; 
$N_H$ is the human population; 
$\eta_H^{-1}$ and $\gamma_V^{-1}$ are the intrinsic and extrinsic incubation periods, respectively; 
$\gamma_{a,d}$ and $\gamma_{a,e}^{-1}$ denote the adult mosquito mortality rate and the expected length of the gonotrophic cycle; 
$n_B$ is the average number of bites a vector gives to hosts during a gonotrophic cycle; 
and $\phi_{VH}$ and $\phi_{HV}$ represent the probability that an infected vector transmits the infection to a human, or acquires the infection from a human, via a single bite. The reproduction number is derived by analyzing the stability of the disease-free state under the assumption of constant temperature. When $R_0 > 1$, the disease-free state becomes unstable and an infected population can grow. In this study, we use $R_0$ as a proxy for virus transmission risk and seek to determine the optimal implementation of mosquito control measures that minimizes $R_0$ over a specified time horizon. It is reasonable to treat $R_0$ as a measure of transmission risk, since its value has a clear interpretation: it represents the expected number of secondary infected mosquitoes produced when a single infected mosquito is introduced into a fully susceptible population of hosts and vectors.
\subsection{Mosquito Population Dynamics} \label{dynm}
Considering the transmission risk defined by Equation~\eqref{reproF}, there are several possible strategies to reduce it, such as controlling the population of adult female mosquitoes, $A$, or decreasing the parameters $\phi_{HV}$, $\phi_{VH}$, and $n_{B}$ by modifying vector--host interactions. However, in this study we focus on control measures that directly target the mosquito population, such as applying larvicide or adulticide and eliminating breeding sites. 

To identify the optimal implementation of a combination of such control measures, we require a high-fidelity mathematical model that accurately describes the mosquito life cycle. The population dynamics of mosquitoes are intrinsically governed by various rates across their life stages, including egg, larval, and pupal development rates, the gonotrophic cycle duration, and the mortality rates in each stage. Each of these rates exhibits strong temperature dependence, which requires careful consideration when developing a mathematical model to describe the population dynamics of mosquitoes. In our previous study~\cite{vajdi2024assessing}, we showed that traditional ordinary differential equation models with temperature-dependent rates do not capture the non-Markovian nature of the biological processes governing the developmental stages of mosquitoes. To accurately describe these processes, it is necessary to employ integro-differential equations rather than standard ODE formulations. However, we showed that such integro-differential equations can be rewritten as a system of ODEs by introducing auxiliary states that record the developmental history at each stage. Specifically, by introducing $J$ sequential substates within each stage of the \textit{Aedes} mosquito life cycle, we can derive a system of ODEs that accurately captures the non-Markovian nature of development across all stages. These ODEs are given by:
\begin{equation}\label{odeF}
{\small
\begin{split}
\frac{1}{J}\frac{d\bm{L}_{j}}{dt}=&f_1\!\left(\sum_i \bm{L}_i, C\right)~\gamma_{e,l}(t)\bm{E}_{J}\ 1_{j=1} \\ &+\gamma_{l,p}(t)\left(1_{j>1}\ \bm{L}_{j-1}-\bm{L}_{j}\right)-\frac{1}{J}\gamma_{l,d}(t)\bm{L}_{j} \\
\frac{1}{J}\frac{d\bm{P}_{j}}{dt}=&\gamma_{l,p}(t)\bm{L}_{J}\ 1_{j=1}+\gamma_{p,a}(t)\left(1_{j>1}\ \bm{P}_{j-1}-\bm{P}_{j}\right) \\ &-\frac{1}{J}\gamma_{p,d}(t)\bm{P}_{j} \\
\frac{1}{J}\frac{d\bm{A}_{j}}{dt}=&\left(\frac{1}{2}\gamma_{p,a}(t)\bm{P}_{J}+\gamma_{a,e}(t)\bm{A}_{J}\right)\ 1_{j=1} \\ &+\gamma_{a,e}(t)\left(1_{j>1}\ \bm{A}_{j-1}-\bm{A}_{j}\right)-\frac{1}{J}\gamma_{a,d}(t)\bm{A}_{j} \\
\frac{1}{J}\frac{d\bm{E}_{j}}{dt}=&ov(t)\gamma_{a,e}(t)\bm{A}_{J}\ 1_{j=1}+\gamma_{e,l}(t)\left(1_{j>1}\ \bm{E}_{j-1}-\bm{E}_{j}\right)\\ &-\frac{1}{J}\gamma_{e,d}(t)\bm{E}_{j} \\
\end{split}
}
\end{equation}
where $j \in \{1,\dots,J\}$, $1_{j=1}$ is an indicator function, 
$\gamma_{e,l}(t)$, $\gamma_{l,p}(t)$, $\gamma_{p,a}(t)$, and $\gamma_{a,e}(t)$ 
are the temperature-dependent development rates of the egg, larval, pupal, and gonotrophic stages, respectively, and 
$\gamma_{e,d}(t)$, $\gamma_{l,d}(t)$, $\gamma_{p,d}(t)$, and $\gamma_{a,d}(t)$ 
are the temperature-dependent mortality rates of eggs, larvae, pupae, and adult mosquitoes. Assuming $C$ is the carrying capacity of the environment, 
$f_1\!\left(\sum_i \bm{L}_i, C\right)$ is a function that decreases as the total larval population increases, becoming zero when the total number of larvae reaches $C$ and equaling one when the larval population is zero. A common choice for such a function is
\[
f_1\!\left(\sum_i \bm{L}_i,\, C\right)
= \left(1 - \frac{\sum_i \bm{L}_i(t)}{C(t)}\right).
\]

In this study, we consider three different control measures that can be implemented to reduce the population of adult female mosquitoes. To simulate the effect of larvicide, we assume that larvicide applications can occur at various times throughout the year. We also assume that the efficacy of the larvicide application is known, and our goal is to determine the optimal timing for a sequence of larvicide applications. We model the effect of larvicide by adding the control-induced mortality to the natural, temperature-dependent larval mortality rate. The following expression is used to represent the control-induced mortality rate:
\begin{equation}
r_{p_l}(t) = \sum_i \sigma_l(t-p_l^i) 
\end{equation}
Here, $\sigma_l$ is a function that defines the profile of each larvicide application, and $p_l^i$ are control parameters representing the timing of each larvicide application. We assume that $\sigma_l$ is continuous with a continuous first derivative. A simple choice for $\sigma_l$ is a function with a fast rising edge that saturates within a few hours, remains approximately constant for several days, and then returns to zero. Figure~\ref{sigmal} illustrates one possible choice for such a function, constructed as the product of two sigmoid functions.
\begin{figure}[H]
    \centering
    \vspace{-1pt}
    \includegraphics[width=0.33\textwidth]{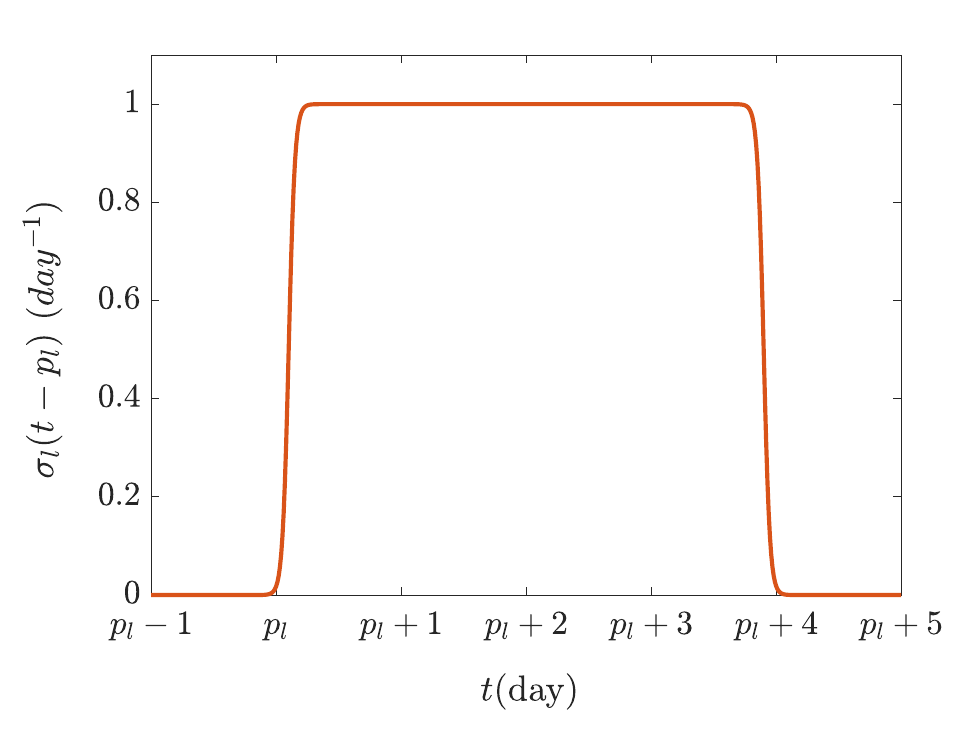}
    \caption{Product of two sigmoid functions used as the time-dependent profile of the induced larval mortality rate for a single application of larvicide.}
    \label{sigmal}
\end{figure}
 
Similarly, we model the effect of adulticide applications by adding a time-dependent mortality term to the natural adult mosquito mortality rate. This term is represented as a sum of pulse functions with different starting times. We note that adulticide efficacy and duration may differ from those of larvicide, and therefore we use a different activation profile, $\sigma_a$, such that the control-induced mortality rate becomes
\begin{equation}
r_{p_a}(t) = \sum_i \sigma_a(t-p_a^i),
\end{equation}
where each $p_a^i$ denotes the timing of the 
$i$-th adulticide application.

In this study, we also consider the control measure of eliminating breeding sites where female mosquitoes lay their eggs. Since \textit{Aedes aegypti} is an urban mosquito species, examples of breeding sites include various water containers such as flower pots, discarded containers, or even potholes that collect water. We model this control measure through its effect on the carrying capacity, $C$, in Equation~\eqref{odeF}. We also account for the fact that, after the elimination of breeding sites, these sites may reappear over time. Therefore, we write the average carrying capacity in an area as  
\begin{equation}\label{cfor}
    C_{p_c}(t) = \big(c_0 - \sum_i \Delta(p_c^i, t)\big)\,\alpha(t),
\end{equation}
where $c_0$ is a constant representing the average number of breeding sites before any control measures are applied, and multiplication by $\alpha(t)$ yields the corresponding carrying capacity after accounting for precipitation effects. The effect of eliminating breeding sites is modeled by subtracting the terms $\Delta(p_c^i, t)$ from the baseline carrying capacity. Each $\Delta(p_c^i, t)$ denotes a campaign aimed at eliminating breeding sites, with a specific efficacy, and $p_c^i$ represents the timing of the intervention. Because eliminated sites may eventually be replenished, the functions $\Delta(p_c^i, t)$ are time-dependent and return to zero after some period. Figure~\ref{deltac} illustrates one possible choice for $\Delta(p_c^i, t)$, constructed as the product of a sigmoid function and an exponential decay with an expected value of 33 days, representing an average reappearance time of eliminated breeding sites.
\begin{figure}[H]
\vspace{-0pt}
    \centering
    \includegraphics[width=0.33\textwidth]{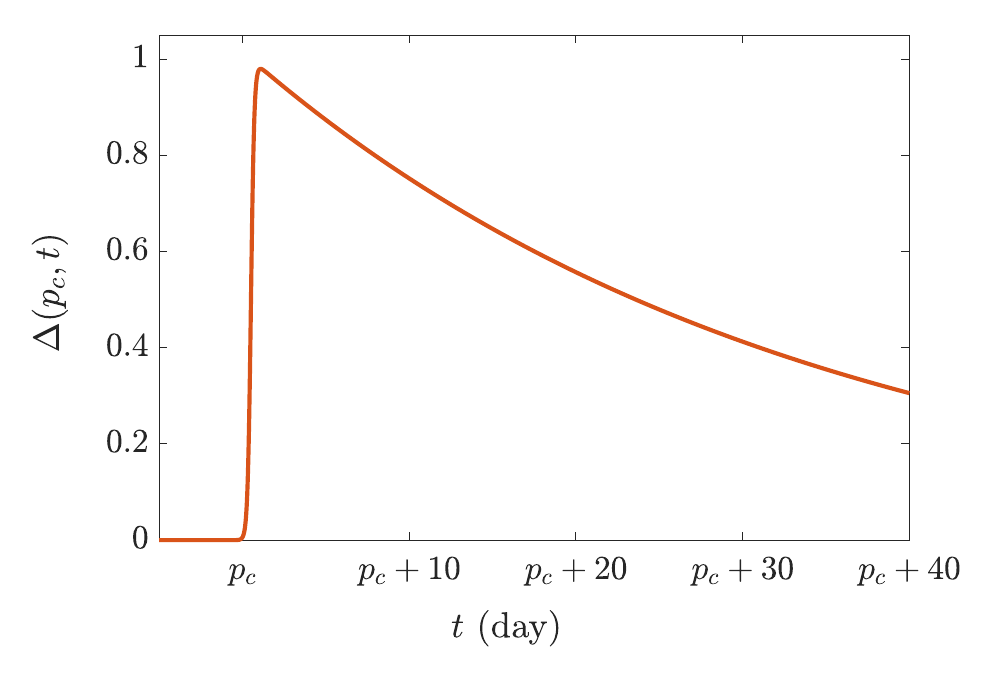}
    \caption{Time-dependent profile of breeding sites after an elimination campaign, with a decay term representing the replenishment of breeding sites.}
    \label{deltac}
\end{figure}

\begin{figure*}[!t]
\centering
\includegraphics[width=\textwidth]{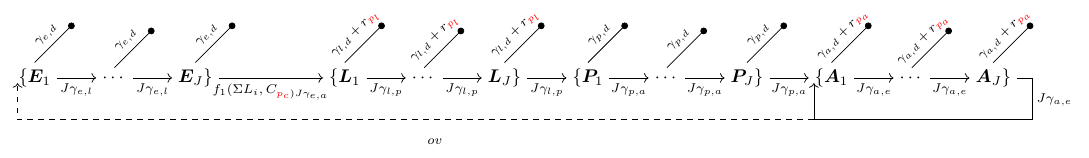}
\caption{Life cycle of the \textit{Aedes} mosquito. The compartments \(\bm{E}\), \(\bm{L}\), \(\bm{P}\), and \(\bm{A}\) denote the egg, larval, pupal, and adult stages, respectively. Each life stage is subdivided into \(J\) sequential substates to more accurately represent the developmental delay required for progression to the next stage. The black circle denotes the dead state. All transition rates, mortality rates, and the number of eggs produced per gonotrophic cycle, \(ov\), are temperature-dependent and therefore time-dependent. The quantities \(r_{p_l}(t)\) and \(r_{p_a}(t)\) denote the additional time-dependent mortality rates induced by larvicide and adulticide application, respectively, and \(C_{p_c}(t)\) denotes the time-dependent reduced carrying capacity resulting from breeding-site elimination. Here, \(p_l\), \(p_a\), and \(p_c\) are the control parameters, representing in particular the timing of the corresponding interventions.}
\label{lifecycleF}
\end{figure*}
In Figure~\ref{lifecycleF}, we show the diagram corresponding to the ODE system~\eqref{odeF} used to simulate the dynamics of the mosquito population. The diagram illustrates the transitions between different life stages and their temperature-dependent development and mortality rates, including the modifications induced by the various control measures.

\subsection{Optimizing Control Measures}\label{opf}
In this study, we aim to determine the optimal scheduling of a combination of mosquito control measures over a given time horizon, with the objective of minimizing the average transmission risk during the same interval. The problem can be formulated mathematically as
\begin{equation} \label{optim}
\begin{split}
\underset{\{p\}}{\text{Minimize}} \quad 
& F = \int_0^{T} R_0\big(\{\mathbf{A}\}, t\big)\, dt \\
\text{subject to} \quad 
& h\big(x, \dot{x}, \{p\}, t\big) = 0,\quad x \in \{\mathbf{A}, \mathbf{P}, \mathbf{L}, \mathbf{E}\}, \\
& g\big(x(0)\big) = 0,
\end{split}
\end{equation}
where $R_0$ is defined in Equation~\eqref{reproF}, and $\{\mathbf{A}\}$ represents the set of states in the ODE system~\eqref{odeF} whose sum yields the adult female population, denoted by $A$ in Equation~\eqref{reproF}. Note that $R_0$ is time-dependent for two distinct reasons. First, several parameters (such as $n_B$) vary with temperature and therefore change over time, giving $R_0$ an intrinsic temperature-driven time dependence. Second, $R_0$ also depends on time through the adult female population itself, since the states in $\{\mathbf{A}\}$ evolve over time according to the ODE system. In problem~\eqref{optim}, $h\big(x, \dot{x}, \{p\}, t\big) = 0$ represents the implicit form of the ODE system~\eqref{odeF}, $\{p\}$ denotes the set of control parameters, and $g\big(x(0)\big)$ specifies the initial conditions.
 
To determine the optimal values of the control parameters $\{p\}$, we perform numerical optimization, which requires computing the total derivative of $F$ with respect to these parameters. These derivatives can be approximated by varying each parameter individually and solving the ODE system for the perturbed value of that parameter. This approach requires solving the ODE system $N_p$ times, where $N_p$ is the number of control parameters. In contrast, the adjoint method~\cite{bryson2018applied,pontryagin2018mathematical} provides an elegant and efficient way to compute all the derivatives simultaneously. In this approach, one first solves the ODE system forward in time to obtain the state variables. Using this solution, a corresponding adjoint ODE system is then formulated and solved backward in time. The adjoint variables quantify how perturbations in the states influence the objective functional $F$, allowing the derivatives $d_{p_i}F$ for all parameters $p_i \in \{p\}$ to be obtained after a single backward integration. Following the adjoint method, the ODE system for the adjoint states is written as
\[
\dot{\boldsymbol{\lambda}}^{T}\, \partial_{\dot{x}} h
= \boldsymbol{\lambda}^{T}\!\left( \partial_x h - d_t \partial_{\dot{x}} h \right)
+ \partial_x R_{0},
\]
where $\boldsymbol{\lambda}$ is the column vector of adjoint states and $\boldsymbol{\lambda}^{T}$ denotes its transpose. The derivation is presented in Appendix~Theorem~\ref{adjointTh}.
 
To explicitly write the adjoint equations, we introduce one adjoint state for each state variable in the ODE system~\eqref{odeF}. For example, for the $j$-th larval state $\boldsymbol{L}_{j}$, we introduce the corresponding adjoint variable $\boldsymbol{\lambda}_{L_j}$. Consequently, we obtain the adjoint ODE system~\eqref{odeB}, which must be integrated backward in time with the terminal condition $\boldsymbol{\lambda}(T) = 0$. Figure~\ref{backward} illustrates a diagram for the adjoint ODEs.
 \begin{figure*}[!t]
\centering
   \includegraphics[width=\textwidth]{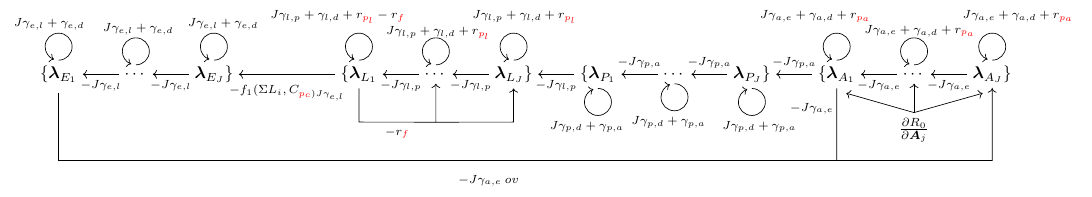}
   \caption{Diagram of the adjoint-state ODE system~\eqref{odeB}. The time derivative of each adjoint state is obtained by summing, over all arrows entering that adjoint state, the product of the variable at the tail of each arrow and the rate associated with that arrow.}  
\label{backward}
\end{figure*}

\begin{equation}\label{odeB}
\small{
\begin{split}
\frac{d\bm{\lambda}_{L_j}}{dt}&=\left(J\gamma_{l,p}(t)+\gamma_{l,d}(t) +r_{p_{l}}\right) ~\bm{\lambda}_{L_j}- 1_{j<J}~J\gamma_{l,p}(t) ~\bm{\lambda}_{L_{j+1}} \\ &  - 1_{j=J} ~J\gamma_{l,p}(t) ~\bm{\lambda}_{P_1} - \frac{\partial f_1\!\left(\sum_i \bm{L}_i, C_{p_c}\right)}{\partial  \bm{L}_j } ~J\gamma_{e,l}(t)~ \bm{E}_J ~\bm{\lambda}_{L_1}\\
\frac{d\bm{\lambda}_{P_j}}{dt}&=\left(J\gamma_{p,a}(t)+\gamma_{p,d}(t)\right) ~\bm{\lambda}_{P_j} - 1_{j<J}~J\gamma_{p,a}(t) ~\bm{\lambda}_{P_{j+1}} \\ &  - 1_{j=J} ~\frac{1}{2}J\gamma_{p,a}(t) ~\bm{\lambda}_{A_1}\\
\frac{d\bm{\lambda}_{A_j}}{dt}&=\left(J\gamma_{a,e}(t)+\gamma_{a,d}(t)+ r_{p_a}\right) ~\bm{\lambda}_{A_j} - 1_{j<J}~J\gamma_{a,e}(t) ~\bm{\lambda}_{A_{j+1}} \\ & - 1_{j=J} ~J\gamma_{a,e}(t) ~\left(\bm{\lambda}_{A_1} + ov(t)\bm{\lambda}_{E_1}\right) +\frac{\partial R_0}{\partial A_j}\\
\frac{d\bm{\lambda}_{E_j}}{dt}&=\left(J\gamma_{e,l}(t)+\gamma_{e,d}(t)\right) ~\bm{\lambda}_{E_j} - 1_{j<J}~J\gamma_{e,l}(t) ~\bm{\lambda}_{E_{j+1}} \\ & - 1_{j=J} f_1\!\left(\sum_i \bm{L}_i, C_{p_c}\right)~J\gamma_{e,l}(t) ~\bm{\lambda}_{L_1}\\
\end{split}
}
\end{equation}
 
After computing the adjoint states, and noting that neither the objective function nor the initial conditions depend on the control parameters, 
the gradient of the objective function $F$ with respect to the optimization parameters can be computed as
\[
d_p F = \int_{0}^{T} \big( \lambda^{T} \partial_{p} h \big)\, dt .
\]
 
Therefore, the gradient with respect to the timing of the $i$th larvicide application, $p_l^i$, is given by
\begin{equation}
\small{
d_{p_l^i} F 
= \int_{0}^{T} \sum_{j=1}^J \bm{\lambda}_{L_j}\, \bm{L}_j \,
\frac{\partial r_{p_l}}{\partial p_l^i} \, dt
= \int_{0}^{T} \sum_{j=1}^J \bm{\lambda}_{L_j}\, \bm{L}_j \,
\frac{\partial \sigma_l(t - p_l^i)}{\partial p_l^i} \, dt ,}
\end{equation}
which can be computed numerically since both the state variables from the forward ODE system and the adjoint states from the backward ODE system are available.
 
Similarly, the gradient with respect to the timing of the $i$th adulticide application, $p_a^i$, is
\begin{equation}
d_{p_a^i} F 
= \int_{0}^{T} \sum_{j=1}^J \bm{\lambda}_{A_j}\, \bm{A}_j \,
\frac{\partial \sigma_a(t - p_a^i)}{\partial p_a^i} \, dt .
\end{equation}
 
Finally, the gradient with respect to the timing of the $i$th habitat elimination action, $p_c^i$, is given by
\begin{equation}
\small{
d_{p_c^i} F 
= \int_{0}^{T} \bm{\lambda}_{L_1}\, \bm{E}_J \,
\frac{\partial f_1\!\left(\sum_i \bm{L}_i, C_{p_c}\right)}{\partial C_{p_c}}
\, \frac{\partial \Delta(p_c^i, t)}{\partial p_c^i}
\, \alpha(t)\, J\gamma_{e,l}(t)\, dt .}
\end{equation}
 
\section{Numerical Results}
In this section, we present numerical results examining the optimal timing of various \textit{Aedes aegypti} mosquito control strategies and their impact on dengue transmission risk. These calculations are carried out using the non-Markovian dynamical model of the mosquito population together with the optimization framework introduced in Sections~\ref{dynm} and~\ref{opf}. To perform these calculations, we use temperature data for Miami-Dade County, Florida. Specifically, we used ERA5 reanalysis temperature data~\cite{hersbach2020era5} for the years 2023 and 2024. For each calendar day, we compute the average temperature for the corresponding days in the two-year dataset. Finally, a one-week moving average is applied to the resulting temperature time series to obtain a smooth seasonal temperature profile. 
 
\subsection{Adulticide Application}\label{ad_sec}
\begin{figure*}[t]
\centering
\subfloat[]{ \includegraphics[width=0.45\textwidth]{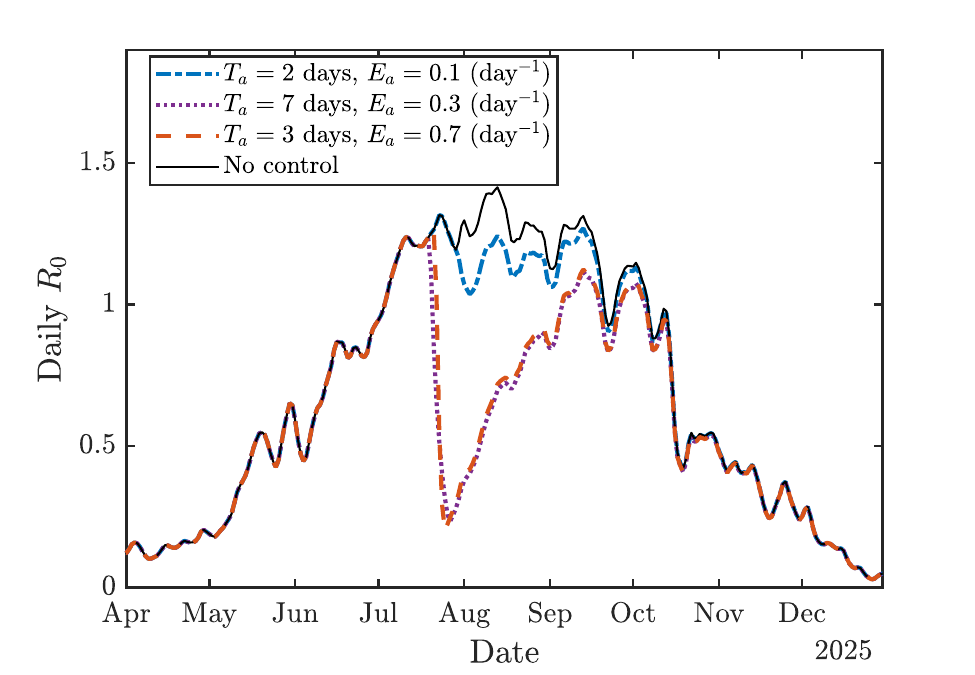} \label{Ic3} }
\subfloat[]{ \includegraphics[width=0.5\textwidth]{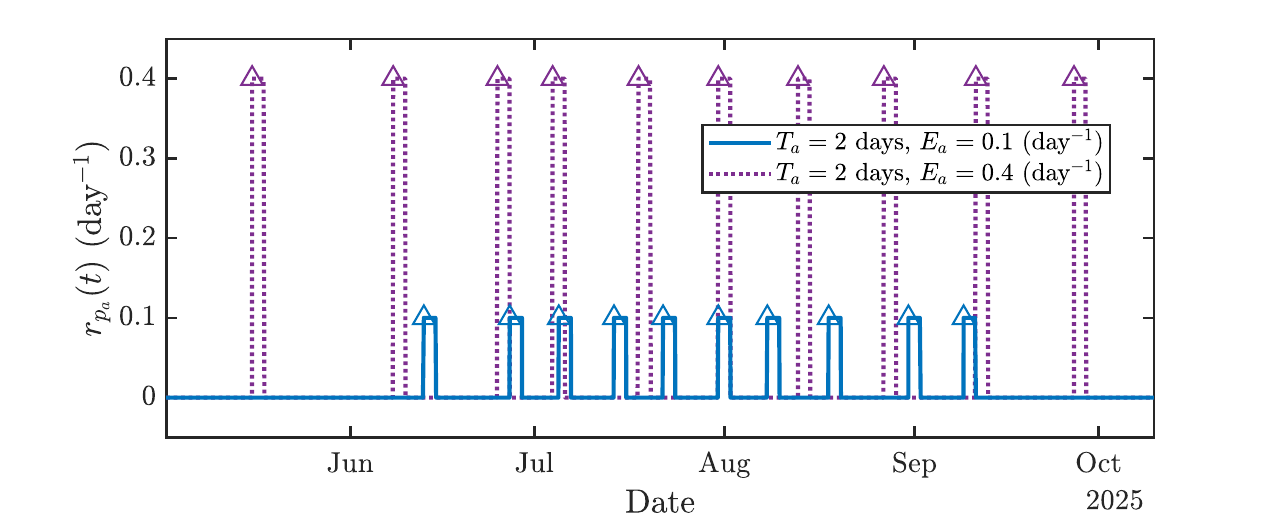}\label{Ic6}}
   \caption{(a) Dengue risk curves under a single optimal adulticide application for different levels of effectiveness and durations.
(b) Optimal timing of a sequence of adulticide applications for two application durations and effectiveness levels.}  
\label{Icc}
\end{figure*}
 
To assess the impact of adulticide application on the mosquito population and dengue transmission risk, we model the overall effect of adulticide intervention in a region through an effective adulticide-induced adult mortality rate, $E_a$, and the duration over which this control-induced mortality is active. Consequently, the additional time-dependent adult mortality rate is given by
\begin{equation}
r_{p_a}(t) = E_a \, \sigma_{T_a}(t - p_a),
\end{equation}
An interesting result emerges when comparing two scenarios: one with parameters 
$T_a = 3~\text{days}$ and $E_a = 0.7~\text{days}^{-1}$, and another with 
$T_a = 7~\text{days}$ and $E_a = 0.3~\text{days}^{-1}$. 
Figure~\ref{Ic3} shows that the reduction in dengue transmission risk is similar in both scenarios. 
This occurs because $R_0$ is a linear function of the adult mosquito population, and 
for relatively short intervention durations $T_a$ compared with the overall system dynamics, the effect of the increased adult mortality rate on the adult population can be approximated by the product $T_a E_a$. Figure~\ref{Ic3} also shows that an increase in the mortality rate of $E_a = 0.1~\text{days}^{-1}$ applied for two days does not significantly decrease the overall dengue transmission risk. Assuming that the natural mortality rate is zero, this corresponds to eliminating approximately $10\%$ of the adult mosquito population per day over a two-day period using adulticide. 
 
We also considered a scenario in which 10 adulticide treatments were implemented. 
We assume that the effectiveness was similar for all 10 treatments and optimize the timing of these interventions. 
In this scenario, the additional induced mortality rate can be written as
\[
r_{p_a}(t) = \sum_i E_a \, \sigma_{T_a}(t - p_a^i),
\]
and we found the optimal set of parameters $p_a^i$. Figure~\ref{Ic6} shows the optimal timing of these treatments. 
The optimization procedure tends to distribute the treatments over several months rather than concentrating them in a short period, which would correspond to a sharp and temporary increase in the mortality rate. 
This behavior arises because, regardless of how aggressively adulticide is applied, the larval population can persist in the environment and subsequently replenish the adult mosquito population. Figure~\ref{Ic8} shows that the reproductive number $R_0$ decreases below one when ten adulticide treatments, each with a duration of two days and effectiveness $E_a = 0.1~\text{days}^{-1}$, are optimally distributed over a period of approximately three months centered around the last week of July.
 
\subsection{Larvicide Application}\label{lar_sec}
\begin{figure*}[t]
\centering
\subfloat[]{ \includegraphics[width=0.45\textwidth]{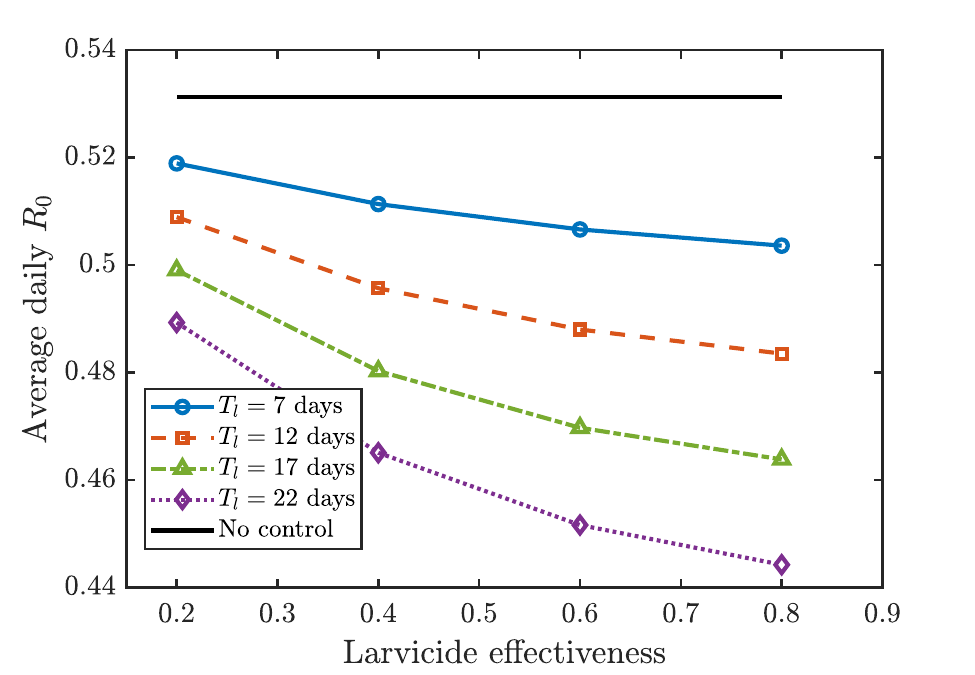}\label{Lar2}} 
\subfloat[]{ \includegraphics[width=0.45\textwidth]{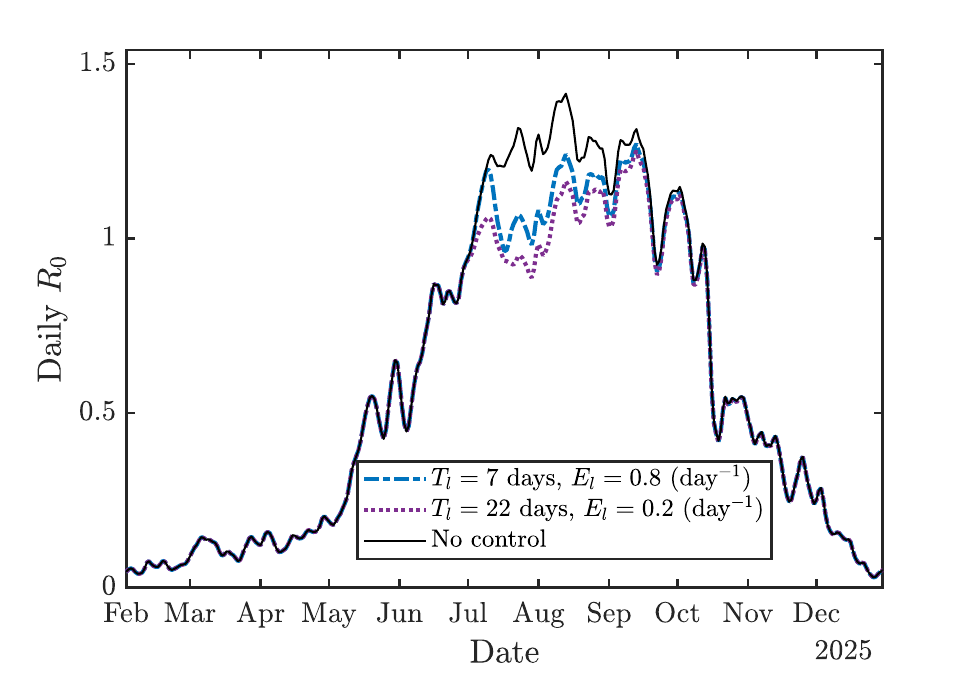} \label{Lar3} }
   \caption{(a) Average daily $R_0$ over the course of a year under a single optimal larvicide treatment with different effectiveness levels and treatment durations.
(b) Dengue risk curves under a single optimal larvicide application for two different effectiveness levels and treatment durations.}  
\label{LL1}
\end{figure*}
Similar to the adulticide treatment, we study the optimal scheduling of larvicide interventions. 
For larvicide, we assume that its effect is represented through an additional time-dependent larval mortality rate defined by
\begin{equation}
r_{p_l}(t) = E_l \, \sigma_{T_l}(t - p_l),
\end{equation}
where $E_l$ denotes the larvicide-induced larval mortality rate, and $\sigma_{T_l}$ is a rectangular function representing the intervention window, with its left edge at $p_l$ and duration $T_l$, similar to the function illustrated in Figure~\ref{sigmal}.
 
The optimization results indicate that the optimal application times $p_l$ consistently occur within approximately one week of July~1. Figure~\ref{Lar2} shows the average daily $R_0$ over the course of a year for different values of $T_l$ and $E_l$, evaluated at the optimized value of $p_l$. Comparing two scenarios, one with parameters $T_l = 7~\text{days}$ and $E_l = 0.8~\text{days}^{-1}$ and another with $T_l = 22~\text{days}$ and $E_l = 0.1~\text{days}^{-1}$, we observe that the reduction in dengue transmission risk is greater in the latter scenario. Despite its lower larval mortality rate $E_l$, the longer intervention period leads to a stronger overall reduction in transmission compared with the shorter but more aggressive treatment. We also note that, in general, our simulations indicate that larvicide treatment leads to a smaller reduction in transmission risk compared to adulticide treatment when similar values are assumed for the intervention duration and effectiveness parameters. This can be seen by comparing the curve corresponding to $T_l = 7~\text{days}$ and $E_l = 0.8~\text{days}^{-1}$ in Figure~\ref{Lar3} with the curve corresponding to $T_a = 7~\text{days}$ and $E_a = 0.3~\text{days}^{-1}$ in Figure~\ref{Ic3}.
 
In addition, we investigated a scenario in which ten larvicide treatments were implemented. Figures~\ref{Lar6} and~\ref{Lar9} illustrate the corresponding time-dependent larvicide-induced larval mortality rate and the resulting temporal evolution of $R_0$. The optimized scenario with $E_l = 0.2~\text{day}^{-1}$ distributes the rectangular control functions over a longer time period, even though the optimization allows these functions to be stacked to produce a more concentrated and aggressive intervention profile. Interestingly, the optimized average daily $R_0$ in this case is lower than that obtained under a more aggressive intervention with $E_l = 0.8~\text{day}^{-1}$ but a shorter duration of effectiveness.
 
\subsection{Optimization of Combined Control Strategies} \label{combsec}
\begin{figure*}[t]
\centering
\subfloat[]{ \includegraphics[width=0.4\textwidth]{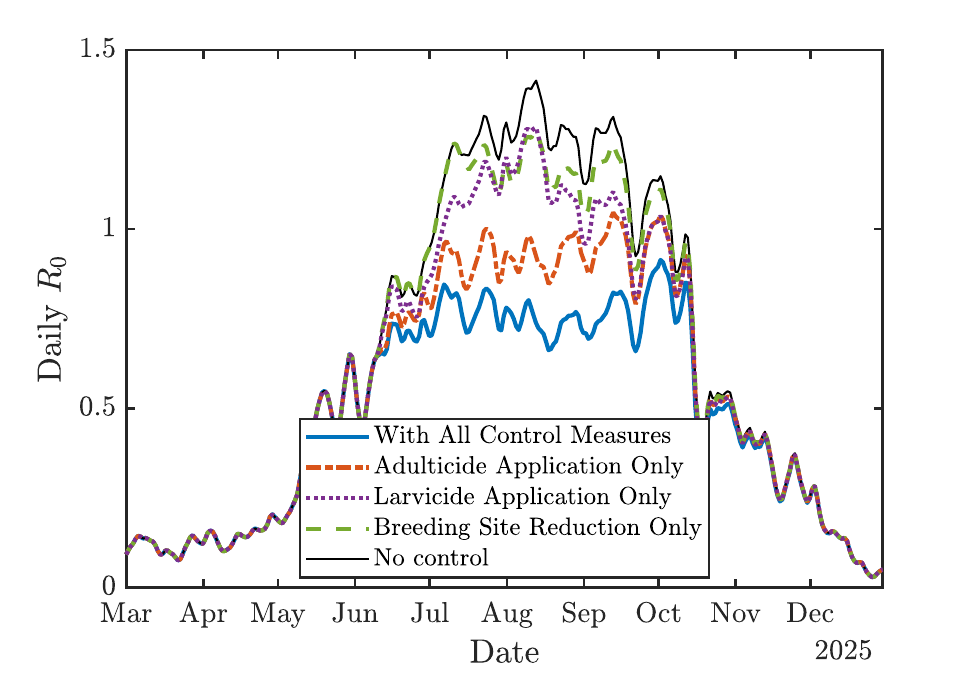} \label{all1} }
\subfloat[]{ \includegraphics[width=0.5\textwidth]{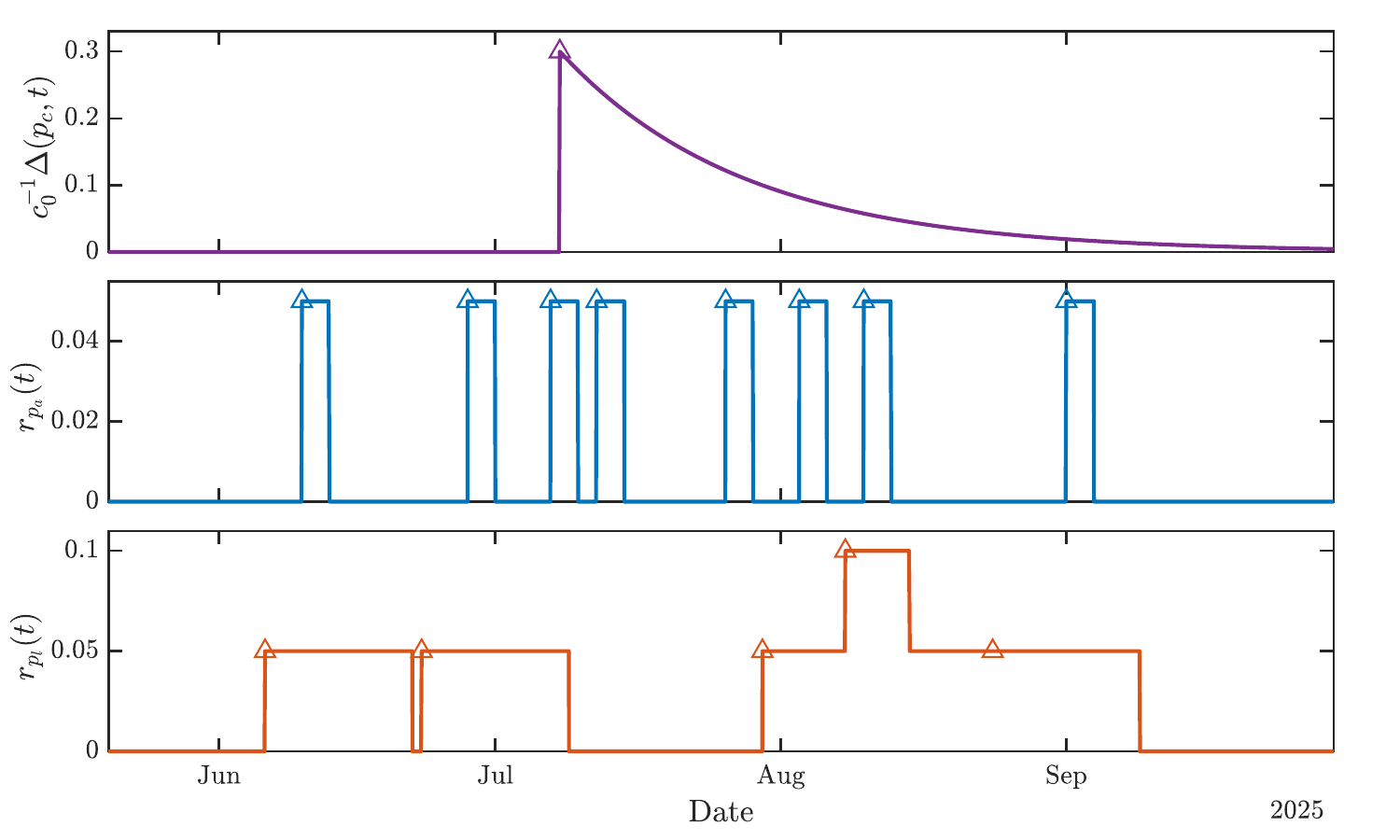}\label{all3}} 
   \caption{(a) Dengue risk curve for the scenario in which eight adulticide treatments, five larvicide treatments, and one breeding-site elimination are optimally applied, as discussed in Section~\ref{combsec}.
(b) Time-dependent profiles of the induced adult mortality rate, larval mortality rate, and breeding-site reduction when the full combination of control measures discussed in Section~\ref{combsec} is applied. This panel shows the optimal timing of each treatment.}  
\label{HH1}
\end{figure*}
To assess the impact of an optimal combination of different control measures, we consider a scenario in which eight adulticide treatments, five larvicide treatments, and one breeding-site elimination campaign are implemented. For this scenario, we assume that the adulticide and larvicide treatments are not aggressive, and the corresponding induced mortality rates can be written as
\begin{equation}
r_{p_a}(t) = \sum_{i=1}^{8} E_a\,\sigma_{T_a}(t-p_a^i), \qquad
r_{p_l}(t) = \sum_{i=1}^{5} E_l\,\sigma_{T_l}(t-p_l^i),
\end{equation}
with $E_a = 0.05~\text{day}^{-1}$, $T_a = 3~\text{days}$, and $E_l = 0.05~\text{day}^{-1}$, $T_l = 17~\text{days}$. 
 
We also assume that the initial reduction in breeding sites is $30\%$, and that the expected recovery time of the breeding habitat is 20 days. Therefore, the resulting change in the carrying capacity, as defined in Equation~\eqref{cfor}, can be written as
\begin{equation}
\Delta(p_c,t) = 0.3\,c_0\,e^{(t-p_c)/20}, \qquad t > p_c .
\end{equation}
For this scenario, we optimize the 14 intervention timing parameters $p_a^i$, $p_l^i$, and $p_c$. To mitigate convergence to local minima, we perform 100 optimization runs using randomly generated initial parameter values. 
 
Figure~\ref{all1} illustrates the effect of the combined control measures as well as the impact of each intervention when applied individually. For the parameter values considered here, the transmission risk is reduced to below one under the combined strategy, with adulticide-based interventions contributing the most to this reduction. Figure~\ref{all3} illustrates the optimal scheduling of the different interventions. 
For the combination of control measures considered here, together with their assumed effectiveness and duration, the optimization procedure distributes the treatments over the months of June, July, and August, corresponding to the peak mosquito season in Florida.
 
In practice, larvicide treatments in this region often begin in late April. However, under the resource constraints assumed in this scenario, the optimization indicates that an optimal strategy would delay the start of larvicide treatments until mid-June. This result is consistent with our earlier simulation in which greater resources were allocated to larvicide interventions. In that case, we considered a scenario with ten larvicide treatment periods, whose optimal scheduling is shown in Figure~\ref{Lar6}. In that scenario, the optimal larvicide treatments begin in late April.

\section{Discussion}
In this work, we presented a framework for optimally scheduling vector control measures. The framework accounts for the complex nature of the mosquito life cycle, which depends on environmental conditions, climate, and habitat availability, and it can incorporate a range of control measures. The optimization is based on the adjoint method and is inherently numerical, which makes the framework flexible and amenable to extension. In particular, it can be adapted to incorporate additional interventions, including biological strategies such as the release of genetically modified or \textit{Wolbachia}-infected mosquitoes, for which timing may also be critical. Therefore, the framework can be applied more broadly to interventions for other vector-borne diseases, such as West Nile virus, to alternative population-dynamics models that may account for spatial heterogeneity, and to more complex intervention settings. However, one especially practical application of this framework is closed-loop vector management, which we discuss in more detail below.

\subsection{Toward Closed-Loop Vector Management via Model Predictive Control}\label{mpc}

Operational vector management is intrinsically an adaptive, feedback-driven process. Integrated Vector Management (IVM) frameworks identify entomological surveillance and continuous monitoring as core operational pillars of evidence-based decision-making~\cite{who2012ivm}. In practice, however, vector-control settings are shaped by substantial uncertainties, including stochastic meteorological forcing (for example, rainfall events that either expand or flush aquatic habitats), fine-scale spatial heterogeneity in habitat density, and localized variation in insecticide efficacy arising from application mechanics, environmental degradation, or physiological resistance~\cite{bowman2016dengue}. Consequently, a static open-loop intervention schedule derived from a single preseason optimization is inherently fragile. In reality, practitioners operate in a closed loop: they apply a treatment, quantify the population response through surveillance, and use this feedback to determine the timing and placement of subsequent interventions~\cite{wilke2021effectiveness}.

Several biological features of the \textit{Aedes aegypti} system make a closed-loop approach especially important. First, interventions act on distinct developmental stages: larvicides target aquatic immatures, adulticides target flying adults, and habitat elimination reduces carrying capacity. As a result, the epidemiological effect of each intervention is distributed over time according to the temperature-dependent developmental delays captured by our non-Markovian model. Second, \textit{Aedes} eggs are resistant to desiccation and can persist through extended dry periods, hatching rapidly upon inundation~\cite{farnesi2009embryonic}; thus, a single unseasonal rainfall event can repopulate aquatic habitats on a timescale that invalidates a schedule optimized under dry-season assumptions. Third, density-dependent larval competition introduces compensatory feedback: reducing larval crowding through larvicide may increase the per-capita survival and fitness of the remaining immatures, partially offsetting the intervention~\cite{barrera1996competition}. Fourth, insecticide resistance profiles vary spatially and can shift over the course of a season, causing realized efficacies to differ from assumed values~\cite{chan2017cdc_surveillance}. Collectively, these mechanisms generate nonlinear, delayed, and environmentally modulated intervention responses whose realized magnitudes cannot be forecast \textit{a priori} with the precision required by open-loop optimality.

To bridge this gap between theoretical optimization and operational reality, the adjoint-based framework developed in this study is well suited for integration into a Model Predictive Control (MPC) architecture, that is, a receding-horizon strategy designed for systems requiring continual feedback and recalibration~\cite{kohler2021robust}. In epidemic management, MPC strategies that periodically re-optimize using updated state estimates have been shown to outperform open-loop policies under model mismatch and measurement noise~\cite{kohler2021robust,peni2020nonlinear}. The same reasoning applies directly to entomological control.

The MPC cycle begins with surveillance and state estimation. At each operational epoch, entomological data are collected~\cite{chan2017cdc_surveillance}, and the observed aggregate counts are allocated across the sequential substates to re-initialize the non-Markovian ODE system. A practical heuristic is to distribute mass proportionally to the equilibrium substate profile at the prevailing temperature, thereby preserving the observed total while assigning developmental age in a manner consistent with the temperature-dependent stage dynamics. When finer-grained observations are available, such as larval surveys that distinguish early from late instars, this allocation can be further constrained. For sparse or noisy trap data, sequential Bayesian data assimilation, for example through Ensemble Kalman Filters or Particle Filters, provides a principled way to update both the latent state and uncertain parameters as new observations become available~\cite{villela2015bayesian}.

Over a rolling horizon, the adjoint method is then used to solve the finite-horizon optimization problem. It is important that the horizon span multiple \textit{Aedes} generations, on the order of four to eight weeks, so that the delayed epidemiological benefits of immature-targeted interventions are properly represented~\cite{eisen2014impact}. The adjoint formulation preserves computational tractability by obtaining the full gradient vector through a single backward integration. In each cycle, vector-control units implement only the immediate segment of the optimized schedule, while later actions are deferred to subsequent updates. After treatment, new surveillance data quantify the realized population response, the horizon is advanced, the model state is updated, and the cycle repeats according to the local surveillance cadence, for example every three days for ovitrap networks or weekly for adult trapping programs~\cite{chan2017cdc_surveillance}.

This architecture aligns directly with IVM operational mandates for continuous monitoring and evaluation~\cite{who2012ivm} and localizes the effect of model mismatch and environmental stochasticity to individual surveillance epochs. This robustness advantage over open-loop control has been established in analogous epidemic-management settings~\cite{kohler2021robust}. The receding-horizon formulation also accommodates evolving meteorological information in a natural way. The open-loop solutions in Section~3 rely on a smoothed historical temperature profile, whereas an operational MPC implementation would use high-resolution short-term weather forecasts for near-term decisions and revert to climatological averages only toward the end of the horizon. Likewise, unanticipated rainfall that restores breeding habitats beyond the recovery dynamics assumed in~\eqref{cfor} would be detected at the next surveillance epoch and incorporated directly into the updated state.

From this perspective, the numerical results in Section~3 may be viewed as a theoretical baseline for the behavior of such an MPC system. Although the open-loop optimizer naturally distributes treatments over several months to avoid rapid population rebound, an MPC controller would generate this temporal spacing adaptively: compressing the schedule when early-season treatments underperform, relaxing it when they overperform, and reallocating finite resources toward periods of highest projected risk. Operational constraints, including budget ceilings, crew availability, chemical inventory, minimum re-treatment intervals, and resistance-management rotation schedules, can be incorporated directly into the finite-horizon problem as inequality constraints on the control parameters, while remaining fully compatible with adjoint-based gradient computation.

An important future extension concerns spatial resolution. The current formulation optimizes temporal scheduling, that is, \textit{when} to treat, within a homogeneously mixed setting. Operational vector control, however, also requires spatial targeting, namely \textit{where} to treat. Because \textit{Aedes aegypti} typically disperses over relatively short distances, often less than 200 meters, urban environments are naturally represented as networks of weakly coupled spatial patches~\cite{chan2017cdc_surveillance}. Coupling our non-Markovian dynamics with a spatially explicit metapopulation model and driving the MPC framework with spatially resolved trap data would yield a decision-support tool capable of dynamically dispatching resources to the urban micro-environments with the highest acute transmission risk.

\ifCLASSOPTIONcompsoc
  \section*{Acknowledgments}
\else
  \section*{Acknowledgment}
\fi
This work has been supported by the United States Department of Agriculture ARS under agreement number 58-3022-3-025.

  \section*{Declaration of Generative AI and AI-assisted technologies in the writing process}
During the preparation of this work the authors used [https://chat.openai.com/] in order to check the manuscript grammar. After using this service, the authors reviewed and edited the content as needed and take full responsibility for the content of the publication.
\newpage
 
 \setcounter{section}{0}
 \renewcommand{\thesection}{\Alph{section}}

 \setcounter{equation}{0}
\setcounter{figure}{0}
\setcounter{theorem}{0}

\renewcommand{\theequation}{\Alph{section}.\arabic{equation}}
\renewcommand{\thefigure}{\Alph{section}.\arabic{figure}}
\renewcommand{\thetheorem}{\Alph{section}.\arabic{theorem}}

\section{Appendices}
\begin{theorem}[Adjoint formula for the gradient] \label{adjointTh}
Let $p$ denote the vector of intervention timing parameters and let
$x(\cdot;p)$ be the solution of the state system
\[
h(x,\dot x,p,t)=0,\qquad g(x(0))=0.
\]
Assume that $h$ is continuously differentiable with respect to
$(x,\dot x,p)$, that the initial condition does not depend on $p$, and
that the objective
\[
F(p)=\int_0^T R_0(x(t;p),t)\,dt
\]
is continuously differentiable with respect to $x$. Let $\lambda$ solve
the adjoint system
\[
\dot{\lambda}^{T}\partial_{\dot x}h
=
\lambda^{T}\big(\partial_x h-d_t(\partial_{\dot x}h)\big)
+\partial_x R_0,
\qquad \lambda(T)=0.
\]
Then
\[
\frac{dF}{dp}
=
\int_0^T \lambda(t)^T\,\partial_p h(x,\dot x,p,t)\,dt .
\]
\end{theorem}

\begin{proof}
Consider the augmented functional
\[
\mathcal{J}(x,p,\lambda)
=
\int_0^T
\Big(
R_0(x,t)+\lambda(t)^\top h(x,\dot{x},p,t)
\Big)\,dt .
\]
Since the state constraint satisfies
\[
h(x,\dot{x},p,t)=0
\]
along admissible trajectories, we have
\[
F(p)=\mathcal{J}(x(\cdot;p),p,\lambda)
\]
for any sufficiently smooth \(\lambda\).

Differentiating with respect to \(p\) gives
\small{
\[
\frac{dF}{dp}
=
\int_0^T
\left(
\partial_x R_0 \,\frac{\partial x}{\partial p}
+
\lambda^\top \partial_x h \,\frac{\partial x}{\partial p}
+
\lambda^\top \partial_{\dot{x}} h \,\frac{\partial \dot{x}}{\partial p}
+
\lambda^\top \partial_p h
\right)\,dt .
\]
}
Using
\[
\frac{\partial \dot{x}}{\partial p}
=
d_t\!\left(\frac{\partial x}{\partial p}\right),
\]
we integrate the third term by parts:
\[
\int_0^T
\lambda^\top \partial_{\dot{x}} h \,\frac{\partial \dot{x}}{\partial p}\,dt
=
\left[
\lambda^\top \partial_{\dot{x}} h \,\frac{\partial x}{\partial p}
\right]_0^T
-
\int_0^T
d_t\!\left(\lambda^\top \partial_{\dot{x}} h\right)
\frac{\partial x}{\partial p}\,dt .
\]
Hence,
\[
\begin{split}    
\frac{dF}{dp}
&=
\int_0^T
\left[
\partial_x R_0
+
\lambda^\top \partial_x h
-
d_t\!\left(\lambda^\top \partial_{\dot{x}} h\right)
\right]
\frac{\partial x}{\partial p}\,dt\\
&+
\int_0^T
\lambda^\top \partial_p h\,dt
+
\left[
\lambda^\top \partial_{\dot{x}} h \,\frac{\partial x}{\partial p}
\right]_0^T .
\end{split}
\]

Now choose \(\lambda\) so that the coefficient of \(\partial x/\partial p\) vanishes, namely,
\[
\dot{\lambda}^\top \partial_{\dot{x}} h
=
\lambda^\top\!\left(\partial_x h - d_t(\partial_{\dot{x}} h)\right)
+
\partial_x R_0 .
\]
Also, since \(\lambda(T)=0\) and the initial condition is independent of \(p\), the boundary term vanishes. Therefore,
\[
\frac{dF}{dp}
=
\int_0^T
\lambda^\top \partial_p h\,dt .
\]
This proves the result.
\end{proof}
\begin{figure}[h]
\centering
\subfloat[]{ \includegraphics[width=1\columnwidth]{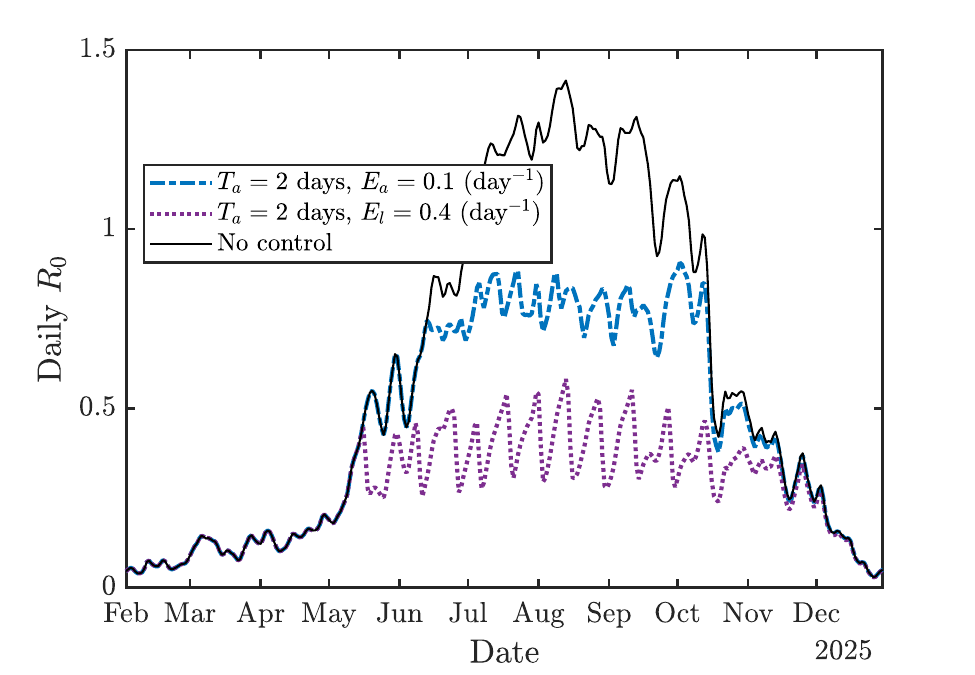}\label{Ic8}}
   \caption{The figure shows the $R_0$ curves for two adulticide-treatment scenarios, each consisting of ten treatments of two-day duration. The two curves correspond to effectiveness levels $E_a = 0.1\,\mathrm{day}^{-1}$ and $E_a = 0.4\,\mathrm{day}^{-1}$, respectively (see section \ref{ad_sec}). }  
\label{Icapp1}
\end{figure}

\begin{figure*}[h]
\centering

\subfloat[]{ \includegraphics[width=0.5\textwidth]{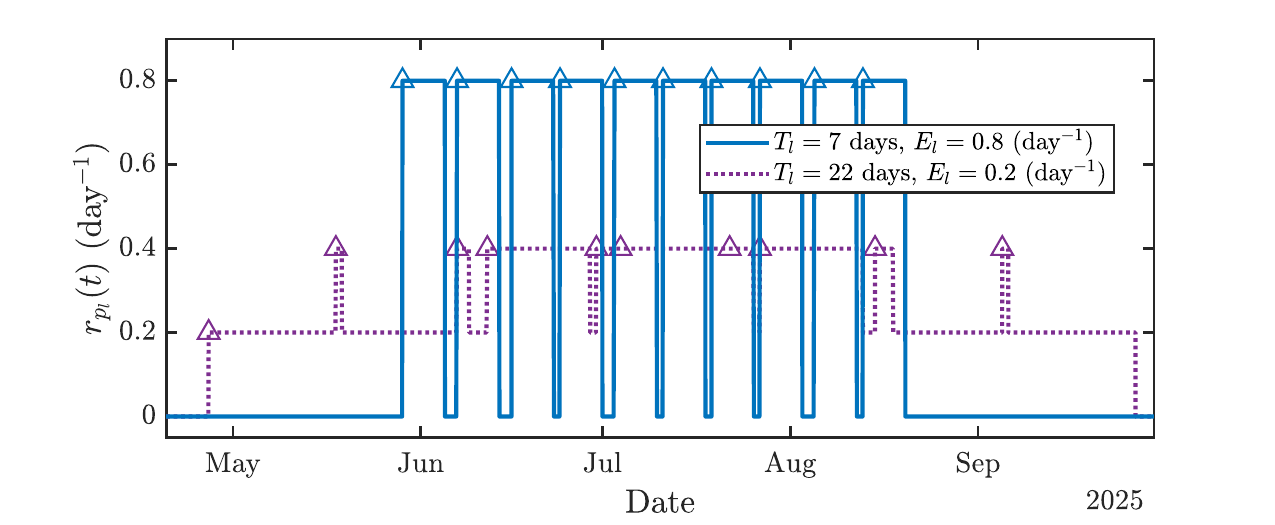}\label{Lar6}}
\subfloat[]{ \includegraphics[width=0.40\textwidth]{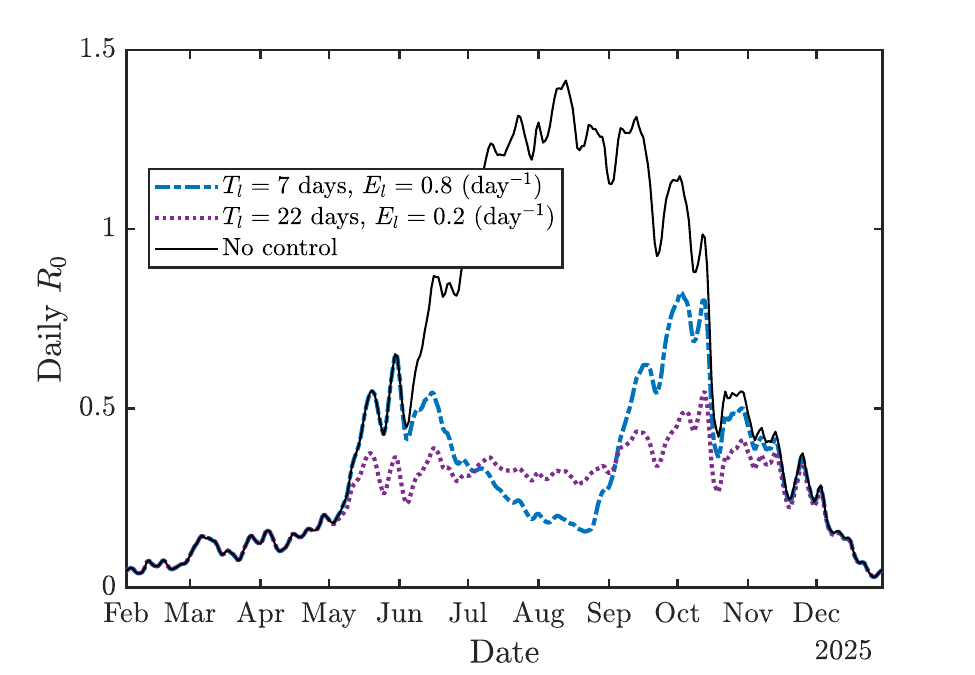}\label{Lar9}}
   \caption{ Panels (a) and (b) show the time-dependent larvicide-induced larval mortality rate associated with ten optimally scheduled larvicide applications and the corresponding temporal evolution of $R_0$, respectively (see Section~\ref{lar_sec}).}  
\label{Larapp1}
\end{figure*}

\newpage
\ifCLASSOPTIONcaptionsoff
  \newpage
\fi

\bibliographystyle{IEEEtran}
\bibliography{Refrences}

%




\end{document}